\title{On simple connectivity of random 2-complexes}
\author{Zur Luria\thanks{Israel Institute for Advanced Studies. e-mail: zluria@gmail.com~.} \and  {Yuval Peled\thanks{Department of Computer Science, Hebrew University, Jerusalem 91904,
    Israel. e-mail: yuvalp@cs.huji.ac.il~.}
}}

\date{\today}
\documentclass[11pt]{article}
\usepackage{amssymb,fullpage}
\usepackage{xspace}
\usepackage{latexsym}
\usepackage{times}
\usepackage{amsfonts}
\usepackage{amsmath}
\usepackage{graphicx}
\usepackage{subcaption}
\usepackage{mathrsfs}
\usepackage{bbm}

\usepackage{verbatim}

\usepackage{amsthm}
\usepackage{amsmath}
\usepackage{amssymb}
\usepackage{graphicx}
\usepackage{epstopdf}
\usepackage{titling}

\usepackage{enumerate}

\usepackage{algpseudocode}
\usepackage{algorithm}
\usepackage{algorithmicx}
\usepackage{tikz}
\usetikzlibrary{shapes.gates.logic.US,trees,positioning,arrows,patterns}

\newcommand{\ignore}[1]{}

\newcommand{\E}{\ensuremath{\mathbb E}}

\newcommand{\F}{\ensuremath{\mathbb F}}

\newcommand{\STIR}[1]{\ensuremath{\left(\frac {#1}e \right)^{#1}\sqrt{2\pi #1}}}

\newtheorem{theorem}{Theorem}[section]

\newtheorem{lemma}[theorem]{Lemma}
\newtheorem{claim}[theorem]{Claim}

\newtheorem{conjecture}[theorem]{Conjecture}
\newtheorem{definition}[theorem]{Definition}

\newtheorem{remark}[theorem]{Remark}

\begin{document}
\maketitle
\begin{abstract}
The fundamental group of the $2$-dimensional Linial-Meshulam random simplicial complex $Y_2(n,p)$ was first studied by Babson, Hoffman and Kahle. They proved that the threshold probability for simple connectivity of $Y_2(n,p)$ is about $p\approx n^{-1/2}$. In this paper, we show that this threshold probability is at most $p\le (\gamma n)^{-1/2}$, where $\gamma = 4^4/3^3$, and conjecture that this threshold is sharp.

In fact, we show that $p=(\gamma n)^{-1/2}$ is a sharp threshold probability for the stronger property that every cycle of length $3$ is the boundary of a subcomplex of $Y_2(n,p)$ that is homeomorphic to a disk. Our proof uses the Poisson paradigm, and relies on a classical result of Tutte on the enumeration of planar triangulations.  
\end{abstract}
\section{Introduction}
The binomial model $Y_2(n,p)$ of random $2$-dimensional simplicial complexes was introduced by Linial and Meshulam more than a decade ago ~\cite{LM}. A complex in this model has $n$ vertices, a complete $1$-dimensional skeleton, and every $2$-dimensional face is included independently with probability $p$. The theory of random complexes studies their typical  properties when the number of vertices $n$ is large. More precisely, for a given function $p=p(n)$, we 
say that $Y_2(n,p)$ has some property {\em asymptotically almost surely}
(abbreviated {\em a.a.s}) if the probability of having this property tends to $1$ as 
$n\to\infty$. We say that a monotone property $P$ has a {\em sharp threshold} at $p=p(n)$ if there exists some negligible function $\varepsilon=\varepsilon(n)\to 0$ for which a.a.s.\ $Y_2(n,(1+\varepsilon)p)$ has the property $P$, and $Y_2(n,(1-\varepsilon)p)$ does not.

For example, Linial and Meshulam showed that 
homological connectivity of $Y=Y_2(n,p)$ over $\F_2$, i.e., $H_1(Y;\F_2)=0$, has a sharp threshold at $p=2\log n/n$. 
This is a $2$-dimensional analog of the classical result by Erd\H{o}s and R\'enyi which asserts that connectivity of a random $G(n,p)$ graph has a sharp threshold at $p=\log n/n$, and it has been generalized further to homologies over other rings and to higher dimensions \cite{MW,HKP,HKP_GAR,LuP}. 

In this paper we are concerned with the stronger notion of \emph{simple connectivity}. A path connected topological space $Y$ is simply connected if its fundamental group vanishes, or in other words, if every closed loop in $Y$ can be contracted to a point. The fundamental group of $Y=Y_2(n,p)$  was first studied by Basbon, Hoffman and Kahle \cite{BHK}. Their main result is that for every constant $\varepsilon>0$, if $p\le n^{-1/2-\varepsilon}$ then $\pi_1(Y)$ is a.a.s hyperbolic and non-trivial. 
In addition, they showed that $Y$ is a.a.s simply connected if $p>\Omega(\sqrt{\log n / n})$. 
Several additional properties and parameters of the random group $\pi_1(Y)$, such as its torsion, cohomological dimension and property (T), have also been studied \cite{CF1,CF2,GW_T,HKP_GAR}. 
This line of research is part of the systematic study of random groups that started with Gromov's density model \cite{gromov,Oli}. The group $\pi_1(Y)$ is closely related to \.Zuk's triangular model \cite{Zuk} and to the fundamental group of other random complexes (See \cite{kahle}).

The challenging problem of determining the sharp threshold probability for simple connectivity of $Y=Y_2(n,p)$ remains open. Nevertheless, progress has been made in bounding this probability from above. The first upper bound in \cite{BHK} was based on the fact that if $p>\Omega(\sqrt{\log n / n})$ then a.a.s the following holds for every two vertices $x,y\in Y$. First, there exists a $2$-dimensional face $xyw\in Y$. Moreover, the graph $G_{xy}=([n]\setminus\{x,y\},E_{xy})$, where $uv\in E_{xy}$ if both faces $uvx,uvy \in Y$, is connected. In consequence, every cycle $x,y,z$ of length $3$ is homotopy equivalent to the contractible cycle $x,y,w$, by following the path from $z$ to $w$ in $G_{xy}$. Since $Y$ has a full $1$-dimensional skeleton, if every cycle of length $3$ is contractible then $Y$ is simply connected. This argument was modified in \cite{GW_SC} by only requiring the graphs $G_{xy}$ to have a sufficiently large connected component. This modification reduced the bound to $p\le O(1/\sqrt{n})$. In \cite{KPS}, a bound of $p\le 1/(2\sqrt{n})$ was derived by analyzing a random process that finds a contractible subcomplex of $Y$ with a full $1$-skeleton. 

Our main result further improves the upper bound on the threshold for simple connectivity. The constant $\gamma:=4^4/3^3$ below appears in Tutte's classical results on enumeration of planar triangulations \cite{tutte} (See Theorem \ref{thm:tutte}).
\begin{theorem}
\label{thm:main1}
Let $n$ be an integer and $c>1/\sqrt{\gamma}$ a constant. Then, $Y_2(n,c/\sqrt{n})$ is a.a.s simply connected.  
\end{theorem}

A triangulation of the $2$-dimensional disk is a simplicial complex that is homeomorphic to the disk. The {\em boundary} of such a triangulation is the image of the disk's boundary under the homeomorphism. Let $Y$ be a simplicial complex. We say that a cycle is {\em triangulated} in $Y$ if it is the boundary of a triangulation $T$ of the disk that is a subcomplex of $Y$. A cycle that is triangulated in $Y$ is also contractible, but not vice versa. Therefore, Theorem \ref{thm:main1} follows directly from the following theorem. 

\begin{theorem}
\label{thm:main}
The sharp threshold probability for the property that every cycle of length $3$ is triangulated in $Y_2(n,p)$ is $p=(\gamma n)^{-1/2}$.
\end{theorem}

This result is obtained by estimations of the first and second moments of the number of triangulated disks in $Y$ with some fixed boundary of length $3$. In a recent paper of the first author and Tessler \cite{LT}, the threshold probability for the appearance of a triangulated $n$-vertex $2$-sphere in $Y=Y_2(n,p)$ was determined by similar moments estimations. However, the difficulties that arise in estimating these moments are quite different in the two problems. Here we consider triangulated disks that are more correlated than the triangulated spheres studied in \cite{LT}. In addition, here we need better probability estimates in order to show that {\em every} cycle of length $3$ is triangulated in $Y$. On the other hand, instead of spheres with $n$ vertices that are studied in \cite{LT}, we can and will restrict to triangulated disks with $k\ll n$ internal vertices. This fact allows a crude analysis of the second moment that evades delicate technical details.


There are classical examples of simplicial complexes that contain contractible cycles which are not boundaries of triangulated disks, e.g., the dunce hat. However, we conjecture that in the $Y_2(n,p)$ model, triangulations of disks determine the threshold for simple connectivity.
\begin{conjecture}
\label{conj:main}
The sharp threshold probability for simple connectivity of $Y_2(n,p)$ is $p=(\gamma n)^{-1/2}$.
\end{conjecture}

The remainder of the paper is organized as follows. In Section \ref{sec:sketch} we overview our proof strategy. In particular, we describe the difficulties in a naive application of the second moment method, explain our approach to overcome them, and derive Theorem \ref{thm:main} from two lemmas that are proved in the subsequent Section \ref{sec:proofs}. We conclude with a number of open questions in Section \ref{sec:open}.

\section{Proof overview}
\label{sec:sketch}
The vertices in the triangulations we consider are always {\em properly labeled}, i.e., no two vertices have the same label. Occasionally, the triangulations we consider are $n$-labeled, i.e., the vertices are labeled by $\{1,...,n\}$, even though the triangulation contains less than $n$ vertices. The triangles or $2$-dimensional faces in a triangulation are referred to as {\em faces}. 

The cycle of length $3$ with vertices $x,y,z$ is denoted by $[xyz]$. For convenience, we refer to a triangulation of the disk whose boundary is the cycle $[xyz]$ as a triangulation of $[xyz]$. In addition, we can and will identify such a triangulation with the set of its $2$-dimensional faces. This is justified since the set of faces completely determines the triangulation. A standard fact in planar graph theory is that a triangulation of a cycle of length $3$ with $k$ internal vertices has $2k+1$ faces. Given two distinct triangulations $T_1\neq T_2$ of the same cycle, we denote $T_1\sim T_2$ if they are not disjoint, i.e., there is a common face to $T_1$ and $T_2$. 

A classical result of Tutte determines the number of planar triangulations of a polygon with a prescribed number of internal vertices. 
Here we only need the enumeration of triangulations of a cycle of length $3$.
\begin{theorem}[\cite{tutte}]
The number $t_k$ of triangulations of a cycle of length $3$ with $k$ labeled internal vertices is equal to
\[
t_k = \frac{6(4k+1)!}{(3k+3)!}.
\]
\label{thm:tutte} 
\end{theorem}
Asymptotically, $t_k=\Theta\left(\frac{k!}{k^{5/2}}\gamma^k\right),$ where $\gamma = 4^4/3^3$.

Given a triangulation of $[123]$ with $k$ labeled internal vertices and integers $4 \leq a_1< \ldots <a_k \leq n$, there is a canonical way to define a proper $n$-labeled triangulation of $[123]$. Namely, one relabels the $i$-th internal vertex in $T$ by $a_i$. Note that this mapping is a bijection, and therefore the number of proper $n$-labeled triangulations of $[123]$ with $k$ internal vertices is $t_k \cdot \binom{n-3}{k}$.

\begin{proof}[Proof of Theorem \ref{thm:main} - lower bound]
Suppose that $Y=Y_2(n,c/\sqrt{n})$ where $c<1/\sqrt{\gamma}$ is a constant. We will show that a.a.s the cycle $[123]$ is not triangulated in $Y$. Let $X$ be the random variable that counts the number of triangulations of $[123]$ in $Y$. We claim that $\E[X]$ is negligible.

According to Theorem \ref{thm:tutte}, there are at most $\gamma^k\cdot k!$ triangulations $T$ of the labeled triangle $[123]$ with $k$ internal vertices, and so the number of proper $n$-labeled triangluations of $[123]$ with $k$ internal vertices is at most $(\gamma n)^k$. In addition, the probability that a triangulation $T$ appears in $Y$ is $p^{2k+1}$, since $|T|=2k+1$.
Therefore,
\[
\E[X] \le \sum_{k=0}^{n-3} (\gamma n)^k \left(\frac{c}{\sqrt{n}}\right)^{2k+1} = 
\frac{c}{\sqrt{n}} \sum_{k=0}^{n-3} (\gamma c^2 )^k = O\left(\frac 1{\sqrt{n}}\right),
\]
where the last equality follows from the assumption that $c < 1/\sqrt{\gamma}.$ 
\end{proof}

We turn to the main result of this paper, which is the proof of the upper bound in Theorem \ref{thm:main}. Suppose that $c>1/\sqrt{\gamma}$ is a constant, $p=c/\sqrt{n}$ and $Y=Y_2(n,p)$. We will prove that a.a.s., every cycle of length $3$ is triangulated in $Y$ by a triangulation that uses $k$ internal vertices, where $k=\log^2 n$. Our approach is to first show that this occurs for a specific cycle $[123]$ with a sufficiently high probability of $1-o(n^3)$, and to then apply a union bound over all cycles of length $3$.

A standard proof strategy for this type of problem is to let $X$ be the random variable that counts the number of  triangulations of $[123]$ with $k$ internal vertices in the random complex $Y=Y_2(n,p)$, and show that $\E[X]\to\infty$ and $\E[X^2]=(1+o_n(1))\E[X]^2$. The conclusion that a.a.s $X>0$ follows immediately from Chebyshev's inequality. 

In the remainder of this section we will describe two obstacles that this approach meets and overview the arguments needed to overcome them. The details of the proof are presented in the subsequent section.

\paragraph{Simple Triangulations.}
The main difficulty in our approach is that the set of triangulations of $[123]$ properly labeled by $\{1,...,n\}$ is too correlated, so the condition $\E[X^2]=(1+o(1))\E[X]^2$ does not hold. Indeed, by a similar computation to the one in the proof of the lower bound in Theorem \ref{thm:main} we see that $$\E[X]^2 = \Theta\left( \frac{(\gamma c^2)^{2k}}{k^{5} n} \right).$$
On the other hand, we have
$$
\E[X^2]=\sum_{T_1,T_2}p^{4k+2-|T_1\cap T_2|}.
$$
We can bound this sum from below by considering only pairs that are either disjoint, or both have the following form. We place an internal vertex $x$, connect it to $1,2$ and $3$, and triangulate $[23x]$ with $k-1$ additional vertices. Thus,
\begin{equation}\label{eqn:EX2}
\E[X^2]\ge \sum_{\substack{T_1\ne T_2 \\ T_1\cap T_2=\emptyset}}p^{4k+2} +
\sum_{x=4}^{n}\sum_{\substack{T_1\ne T_2 \\ T_1\cap T_2\supseteq \{12x,13x\}}}p^{4k}.
\end{equation}
Clearly, the first term contributes $(1+o(1))\E[X]^2$ since almost all pairs of triangulations are disjoint. In addition, the second term contributes
\[
(n-3)\cdot \Omega \left(\frac{(\gamma n)^{k-1}}{(k-1)^{5/2}}\right)^2 \left(\frac{c}{\sqrt{n}}\right)^{4k}=\Omega \left( \E[X]^2\right),
\]
since the number of triangulations of $[23x]$ with $k-1$ internal vertices is $\Omega\left( (\gamma n)^{k-1}/(k-1)^{5/2} \right).$
Consequently, $\E[X^2]\ge (1+\Omega(1))\E[X]^2$.

Therefore, a naive application of the second moment method fails. Our solution is to \textbf{restrict} the set of triangulations.

Let $T$ be a triangulation of $[123]$. A {\em missing face} of $T$ is a triple of vertices $\{x,y,z\}$ of $T$, which is not $\{1,2,3\}$, such that $xyz$ is not a face of $T$ but $xy,xz,yz$ are edges in $T$. Triangulations with no missing faces are called {\em simple} and were studied by Tutte \cite{tutte}. Here we consider a relaxation of this notion.
The {\em density} of a missing face $\{x,y,z\}$ in $T$ is the number of vertices that lie in its interior.
For example, in the triangulations we considered in the second term of (\ref{eqn:EX2}), $\{2,3,x\}$ is a missing face of density $k-1$. 

\begin{definition}
\label{def:simple}
Let $l$ be an integer. A triangulation $T$ is called {$l$-simple} if the density of every missing face of $T$ is at most $l$. In other words, no missing face contains more than $l$ vertices in its interior.
\end{definition}

Simple triangulations are $0$-simple in our terminology, and Tutte showed that the number of simple triangulations with $k$ labeled internal vertices is roughly $k!(3^3/2^2)^k$, which is exponentially smaller than the number of all triangulations $t_k\approx k!\gamma^k$. In comparison, the following lemma shows that a milder restriction to $k/2$-simple triangulations only decreases their number by a constant factor.
\begin{lemma}
\label{lem:manysimple}
Let $k$ be a large integer. The number of $\frac{k}{2}$-simple triangulations of $[123]$ with $k$ internal labeled vertices is $
\Theta\left(\frac {k!}{k^{5/2}}\gamma^k\right).
$
\end{lemma}

The main ingredient in the proof of Theorem \ref{thm:main} is Lemma \ref{lem:key_lemma} below, which shows that triangulations with dense missing faces are the only obstacles to apply the second moment method here.

\begin{lemma}
\label{lem:key_lemma}
Let $n,k$ be integers such that $k=\lceil\log^2n\rceil$, $c>1/\sqrt{\gamma}$ a constant, and $\mathcal T$ be the set of $\frac{k}{2}$-simple triangulations of $[123]$ with $k$ internal vertices properly labeled by $\{1,...,n\}$. For every $T\in\mathcal T$, denote by $B_T$ the event that $T\subseteq Y$, where $Y=Y_2(n,c/\sqrt{n})$. In addition, suppose that
\[
\mu := \sum_{T\in \mathcal T} \Pr[B_T].
\]
and 
\[
\Delta :=\sum_{\substack{T_1,T_2\in\mathcal T \\ T_1\sim T_2}} \Pr[B_{T_1}\wedge B_{T_2}].
\]
Then, $$\frac{\Delta}{\mu^2} \leq \frac{(\log n)^{O(1)}}{\sqrt{n}}.$$
\end{lemma}

\paragraph{The Poisson paradigm.}
By the second moment method, Lemma \ref{lem:key_lemma} implies that the cycle $[123]$ is triangulated in $Y$ with high probability of at least $1-\frac {(\log n)^{O(1)}}{\sqrt{n}}$. However, this bound on the probability is not high enough to imply Theorem \ref{thm:main}. A possible approach to improve the lower bound on this probability is to restrict the set of triangulations even further by forbidding squares or other shapes with many internal vertices. We take a different path and observe that the random variable $Z = \sum_{T\in\mathcal T}1_{B_T}$, that counts the number of $k/2$-simple triangulations of $[123]$ in $Y$, falls under the framework of the Poisson paradigm (See \cite{alon_spencer}). Roughly speaking, $Z$ is the sum of many, rare, almost independent, indicator random variables. Hence, the probability that $Z=0$ decays exponentially in $\mu^2/\Delta$, rather than polynomially.

\begin{theorem}[Janson Inequality]
\[\Pr\left[\bigwedge_{T\in\mathcal T}\overline{B_T}\right]\le
e^{-\mu/2}+e^{-\mu^2/2\Delta}.\]
\end{theorem}

\begin{proof}[Proof of Theorem \ref{thm:main} - upper bound.]
By Lemma \ref{lem:key_lemma}, $\mu^2/2\Delta \ge \sqrt{n}/(\log{n})^{O(1)}$. In addition, $$\mu/2 \ge \frac{(\gamma c^2)^k}{k^{O(1)}\sqrt{n}} \ge n^{\Omega(\log n)}.$$ Indeed, the first inequality follows from Lemma \ref{lem:manysimple} which asserts that the number of $k/2$-simple triangulations with $k$ internal vertices is $k!\gamma^k/k^{O(1)}$. For every such triangulation there are $(1-o(1))n^k/k!$ proper $n$-labellings, and the probability for each of these $n$-labeled triangulations to appear in $Y$ is $p^{2k+1}$. The second inequality holds since $\gamma c^2>1$ and $k= \log^2 n$.
As a result, Janson inequality implies that the probability that $[123]$ is not triangulated in $Y$ is $o(1/n^3)$. The proof is concluded by replacing $[123]$ with any fixed cycle of length $3$ and taking a union bound.
\end{proof}

\section{Proof of Lemmas \ref{lem:manysimple} and \ref{lem:key_lemma}}
\label{sec:proofs}
\subsection{Many $k/2$-simple triangulations}
We start with the proof of Lemma \ref{lem:manysimple}. Recall that $t_k$ denotes the number of triangulations of $[123]$ with $k$ labeled internal vertices. A triangulation is called  {\em $j$-dense} if it contains a missing face of density $j$, i.e., a missing face with $j$ internal vertices. 
Note that the interiors of two missing faces are either disjoint, or one is contained in the other.
Hence, if $j>k/2$ then there is at most one missing face of density $j$, that is called the {\em $j$-dense missing face} of the triangulation. Denote by $t_k^j\le t_k$ the number of $j$-dense triangulations. 

Let $i:=k-j$. We observe that 
\begin{equation}t_k^j = \binom ki t_it_j(2i+1),
\label{eqn:dnsjcnst}
\end{equation}
for  $j>k/2$. Indeed, in order to construct a $j$-dense triangulation we (i) triangulate $[123]$ with $i$ internal vertices, (ii) select one of the $(2i+1)$ faces and (iii) triangulate its interior with additional $j$ vertices. In addition, we need to choose the labels of the $i$ vertices out of the $k$ that will be used for the initial triangulation. One can verify that this construction is bijective for $j>k/2$

\begin{claim}
\label{clm:tjk}
Let $i,j$ and $k$ be integers such that $k/2<j<k$ and $i+j=k$. In addition, denote
\[
\alpha_i:=\frac{6(2i+1)}{(3i+3)(3i+2)}\binom {4i+1}{i} \gamma ^{-i}.
\]
Then,
\[
\frac{t_k^j}{t_k} = (1+o_k(1))
\alpha_i\left(\frac kj \right)^{5/2}.
\]
\end{claim}
\begin{proof}
By Tutte's formula and (\ref{eqn:dnsjcnst}),
\begin{align*}
\frac{t_k^j}{t_k}~=~&
\frac {6\cdot k!(4i+1)!(4j+1)!(3k+3)!(2i+1)}{j!i!(3i+3)!(3j+3)!(4k+1)!}\\
~=~&
\alpha_i\gamma^i\cdot
\frac {k!(4j+1)!(3k+3)!}{j!(3j+3)!(4k+1)!}\\~=~&
\alpha_i\gamma^i\cdot (1+o_k(1))\left(\frac kj \right)^{2}\frac {\binom {4j}j} {\binom{4k}k }.
\end{align*}
The claim follows immediately from Stirling's approximation $m!=(1+o_m(1))\STIR{m}~$ which implies that $\binom {4m}m = (1+o_m(1))\gamma^m\cdot\sqrt{\frac{2}{3\pi m}}.$
\end{proof}
Note that ${t_k^j}/{t_k}=(1+o_k(1))\alpha_i$ if $i
\ll k$, and ${t_k^j}/{t_k}\le 6\alpha_i$ always holds if $j>k/2$.
In addition, by Stirling's approximation, 
$\alpha_i \le \frac{16\sqrt{2}}{9\sqrt{3\pi}}\cdot{\frac 1{i\sqrt{i}}}.$

\begin{proof}[Proof of Lemma \ref{lem:manysimple}]
The sum $\sum_{j>k/2} t_k^j$ is an upper bound on the number of triangulations that are not $k/2$-simple, but it may not be tight, because a triangulation can be of more than one density. 

We say that a $j$-dense triangulation $T$ is {\em $j$-nested} if the internal triangulation of the $j$-dense missing face of $T$ is of density $j-1$. Clearly, a $j$-nested triangulation is both $j$-dense and $(j-1)$-dense and hence contributes to both corresponding summands.

Let $\hat t_k^j \le t_k^j$ denote the number of $j$-dense triangulations that are {\em not} $j$-nested.
Consider a $j$-dense triangulation $T$, where $j>k/2$, and let $j'\le j$ be the smallest integer that is greater than $k/2$ for which $T$ is $j'$-dense. 
If $j'>\lfloor k/2 \rfloor + 1$ then, by the minimality of $j'$, $T$ is $j'$-dense but not $j'$-nested. Otherwise, $j'=\lfloor k/2 \rfloor + 1$, hence $T$ is $(\lfloor k/2 \rfloor + 1)$-dense.

In consequence, the fraction of triangulations that are not $k/2$-simple is bounded from above by
\begin{align}\label{eqn:sum} 
& \frac{1}{t_k}\left(
t_k^{\lfloor k/2 \rfloor + 1}+
\sum_{ \lfloor k/2 \rfloor + 1 < j < k} \hat t_k^j
\right) \nonumber \\
\leq & \frac{1}{t_k}\left(
\sum_{\frac k2 < j < k-A} t_k^j+
\sum_{ k-A \le j < k-6} \hat t_k^j+
\sum_{ k-6 \le j < k} \hat t_k^j
\right),
\end{align}
where $A$ is some large constant. We need to show that this fraction is bounded away from $1$.

The number of $j$-nested triangulations is
$
\binom kj t_i (t_j^{j-1}) (2i+1)
$
if $j>k/2$.
Indeed, one may construct a $j$-nested triangulation by constructing a $j$-dense triangulation as in (\ref{eqn:dnsjcnst}), where  the internal triangulation of the $j$-dense missing face is restricted to be $(j-1)$-dense. Therefore, $\hat t_k^j =(1+o_k(1))(1-\alpha_1)t_k^j$ by Claim \ref{clm:tjk}.

We bound the three sums in (\ref{eqn:sum}) separately.
\begin{enumerate}
\item A direct calculation shows that the contribution of the sum over $k-1 \ge j \ge k-6$ is $$(1+o_k(1))(1-\alpha_1)\sum_{i=1}^{6} \alpha_i,$$ which is smaller than $0.53$ for a sufficiently large $k$.
\item The sum over $k-7\ge j \ge k-A$ contributes at most
\[
(1+o_k(1))(1-\alpha_1)\sum_{i=7}^{A}\alpha_i \le 
(1+o_k(1))(1-\alpha_1)\sum_{i=7}^{A}\frac{16\sqrt{2}}{9\sqrt{3\pi}}\cdot{\frac 1{i\sqrt{i}}}.
\]
This is smaller than $(1-\alpha_1)\frac{16\sqrt{2}}{9\sqrt{3\pi}}\cdot \frac{2}{\sqrt{6}}\le 0.46
$
for a sufficiently large $k$.
\item The third sum is at most $\sum_{i=A}^{\infty} 6\alpha_i \le O(1/\sqrt{A})$, which can be made arbitrarily small.
\end{enumerate}
Therefore, the sum in (\ref{eqn:sum}) is bounded away from $1$.
\end{proof}

\subsection{$k/2$-simple triangulations are mostly independent}
We turn to the proof of Lemma \ref{lem:key_lemma} and start with a standard trick.
\begin{claim}
\label{clm:trick}
$$\frac{\Delta}{\mu^2} \le \frac{1}{|\mathcal T|^2} \sum_{\substack{S\subset \binom {[n]}3\\0<|S|<2k+1}} p^{-|S|}|\{T\in \mathcal T~:~T\supset S \}|^2$$
\end{claim}
\begin{proof}
First, we recall that $\mu = |\mathcal T| p^{2k+1}$ and 
$\Delta = \sum_{T_1\sim T_2} p^{4k+2-|T_1\cap T_2|}.$ Therefore,
\begin{align*}
\frac{\Delta}{\mu^2} ~=~& \frac{1}{|\mathcal T|^2}\sum_{T_1\sim T_2} p^{-|T_1\cap T_2|} \\
~\le~& \frac{1}{|\mathcal T|^2}\sum_{T_1\sim T_2}\sum_{\substack{S\subseteq T_1\cap T_2\\ \emptyset\neq S} } p^{-|S|} \\
~\le~& \frac{1}{|\mathcal T|^2}\sum_{\substack{S\\0<|S|<2k+1}}p^{-|S|}|\{T\in \mathcal T~:~T\supset S \}|^2.
\end{align*}
The first inequality follows from the fact $T_1\cap T_2 \neq\emptyset$ if $T_1\sim T_2$ so the summation only includes additional elements. The second inequality is obtained by changing the order of summation. The set $S$ cannot be of size $2k+1$ because $T_1\neq T_2$. In addition, the number of pairs $T_1\sim T_2$ for which $T_1\cap T_2 \supseteq S$ is bounded by the square of the number of triangulations in $\mathcal T$ that contain $S$.
\end{proof}
We introduce some notations and terminology for the remainder of the proof (See Figure \ref{fig:example} for an example).
A nonempty set of $2$-dimensional faces, or triples, over $n$ vertices is called {\em admissible} if it is strictly contained in (the set of faces of) some triangulation $T\in\mathcal T$. Let $S$ be an admissible set of faces. We consider a $2$-dimensional simplical complex $X_S$ that is comprised of $S$, edges and vertices that are contained in some face of $S$, and the cycle $[123]$.

The {\em degree} of an edge of $X_S$ equals the number of faces of $S$ that contain it. We denote by $e_l$ the number of edges of degree $l$, where $l=0,1,2$. Note that no edge of $S$ is of degree greater than $2$ since $X_S$ is planar. In addition, $e_0\le 3$ since only the edges of the cycle $[123]$ can have degree $0$. We denote by $V(S)$ the vertex set of $X_S$ and partition it into three subsets: (a) The vertices $1,2$ and $3$ are called the {\em fixed} vertices, (b) the {\em boundary} vertices $V_\partial(S)$ is the set of all other vertices that are contained in an edge of degree $1$, and (c) the remaining vertices $V_I(S)$ are called {\em internal} vertices. We denote by $v_\partial$ and $v_I$ the number of boundary and internal vertices respectively. In particular $|V(S)|=v_\partial +v_I+3$. Any triangulation $T\in\mathcal T$ that contains $S$ has additional $v_O:=k-v_\partial-v_I$ vertices that are called the {\em outer} vertices. 

In addition, we denote by $\beta_0$ the number of connected components of $X_S$ and by $\beta_1$ its first Betti number. The parameter $\beta_1$ has a concrete description that we will use. Namely, consider any realization of the complex $X_S$ in the plane such that the cycle $[123]$ bounds the outer face. $\beta_1$ is the number of connected components in the complement of $X_S$, except for the outer face. In fact, the first homology is generated by the cycles in $X_S$ that bound every such component from the outside. 

We define an additional parameter $\Phi$ that will play a crucial role in the computation by \begin{equation}
\label{eqn:phi}
\Phi:=3-e_0-\frac{e_1}2-\beta_0+\beta_1.
\end{equation}

\begin{figure}
\begin{center}
\begin{tikzpicture}
\draw[fill=white,dashed] (0,1) to (-2,-2.5) to (0.3,0) to (0,-2.2) to (-1.1,-2.1) --(3,-3)--(0,1);

\draw[fill=white] (-3,-3) to (-2,-2.5) to (0,-2.2) to(-3,-3);

\draw[fill=gray!40] (.5,-1.9) to (1.2,-1.5) to (.5,-1) to (.5,-1.9);
\draw[fill=gray!40] (0,1) to (-3,-3) to (-2,-2.5) to (0,1);
\draw[fill=gray!40]  (-3,-3) to (-0,-2.2) to (1.5,-2.3)to (3,-3)to(-3,-3);
\draw[fill=gray!40] (0,-2.2) to (0.3,0) to (-2,-2.5) to (0,-2.2);
\draw[fill=white] (-0.5,-1.8) to (-.1,-.8) to (-1.1,-2.1) to (-0.5,-1.8);

\node[circle,draw,fill=gray!20,inner sep=0.2pt] (1) at (0,1){\footnotesize $1$};
\node[circle,draw,fill=gray!20,inner sep=0.2pt] (2) at (-3,-3){\footnotesize $2$};
\node[circle,draw,fill=gray!20,inner sep=0.2pt] (3) at (3,-3){\footnotesize $3$};
\node[circle,draw,fill=gray!20,inner sep=0.2pt] (4) at (-0,-2.2){\footnotesize $4$};
\node[circle,draw,fill=gray!20,inner sep=0.2pt] (5) at (0.3,0){\footnotesize $5$};
\node[circle,draw,fill=gray!20,inner sep=0.2pt] (6) at (-2,-2.5){\footnotesize $6$};
\node[circle,draw,fill=gray!20,inner sep=0pt] (7) at (1.5,-2.3){\scriptsize $A$};
\node[circle,draw,fill=gray!20,inner sep=0.2pt] (8) at (-0.5,-1.8){\footnotesize $8$};
\node[circle,draw,fill=gray!20,inner sep=0.2pt] (9) at (-.1,-.8){\footnotesize $9$};
\node[circle,draw,fill=gray!20,inner sep=0.2pt] (10) at (-1.1,-2.1){\footnotesize $7$};
\node[circle,draw,fill=gray!20,inner sep=0pt] (11) at (.5,-1.9){\scriptsize $B$};
\node[circle,draw,fill=gray!20,inner sep=0pt] (12) at (1.2,-1.5){\scriptsize $C$};
\node[circle,draw,fill=gray!20,inner sep=0pt] (13) at (.5,-1){\scriptsize $D$};
\node[circle,draw,fill=gray!20,inner sep=0pt] (14) at (0.75,-2.7){\scriptsize $E$};

\draw (3)--(1)--(2)--(4)--(6)--(2)--(3)--(7)--(4)--(5)--(6)--(1);
\draw (8)--(9)--(10)--(8);
\draw (2)--(14)--(4);
\draw (3)--(14)--(7);
\draw (5)--(9);
\draw (4)--(8);
\draw (9)--(6)--(10);
\draw (9)--(4)--(10);
\end{tikzpicture}
\end{center}
\caption{
A planar embedding of $X_S$ induced by the admissible set $$S=\{ 126,23E,24E,3AE,459,467,478,489,4AE,569,679,BCD \}.$$ Here we use hexadecimal labels for convenience. $S$ has $v_\partial = 10$ boundary vertices, $v_I=1$, $e_0=1$, $e_1=16$ and $e_2=10$. Moreover, $S$ has $2$ connected components and $\beta_1(X_S)=3$ since the first homology is generated by the cycles $[246],[789]$ and $[1654A3]$. Therefore, $\Phi = -5$ and $2(v_\partial+v_I+\Phi)=12=|S|$ as predicted by Claim \ref{clm:phieuler} item 2.}
\label{fig:example}
\end{figure}
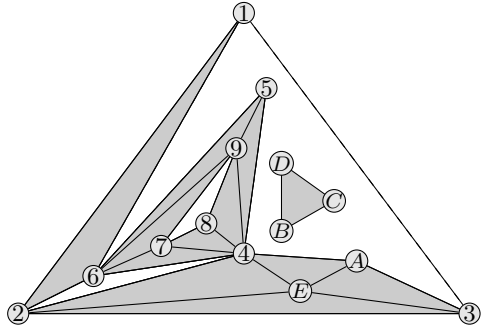

Consider the {\em tuple of parameters} $P=P(S):=(v_\partial,v_I,v_O,\Phi)$ of an admissible set $S$. The basic idea of the proof is to join together all the sets $S$ with the same tuple of parameters $P(S)$. 
Let $\mathcal P:=\{P(S)~:~\mbox{$S$ is admissible}\}$ be the set of all {\em admissible tuples of parameters}. 
\begin{claim}
\label{clm:phieuler}
Let $P=(v_\partial,v_I,v_O,\Phi)\in\mathcal P$ be an admissible tuple of parameters.
\begin{enumerate}
\item One of the following holds. Either (i) $\Phi \le -1/2$ or (ii) $\Phi = 0$ and $v_\partial + v_I\ge k/2$.
\item For every admissible set $S$ such that $P(S)=P$ there holds 
$|S| = 2(v_\partial +v_I + \Phi).$
\item $\Phi \le 1 - \frac{v_\partial}6$.
\end{enumerate}
\end{claim}
\begin{proof}
$1.$ We start by showing that $\Phi \le 0$. Let $S$ be any admissible set whose tuple of parameters is $P(S)=P$, and consider any planar embedding of $X_S$ in which $[123]$ bounds the outer face. 
First note that $\beta_0,\beta_1\ge 1$. Indeed, the image of $X_S$ is a nonempty planar domain that contains the boundary of a triangle $[123]$ but not its entire interior.

Recall that the first Betti number $\beta_1$ is equal to the number of connected components in the complement of $X_S$ minus the outer face. Every edge of degree $0$ or $1$ in $S$ is incident to such a component, except the $(3-e_0)$ edges of degree $1$ between the fixed vertices that are incident to a component of $X_S$ and to the outer face. On the other hand, every component in the complement is bounded by at least $3$ edges of $S$. Therefore,
\begin{equation}\label{eqn:white}3\cdot \beta_1\le e_0+e_1-(3-e_0)=2e_0+e_1-3.\end{equation} In consequence,
\[
\Phi = 3 - e_0-\frac{e_1}{2}-\beta_0+\beta_1
\le
3-\frac 32(\beta_1+1)-1+\beta_1=\frac{1-\beta_1}{2}\le 0.
\]

Therefore, $\Phi = 0$ if and only if $\beta_1=1$ {\em and} the boundary of the unique component of $X_S$'s complement contains exactly $3$ edges. In such a case, in any triangulation $T\in\mathcal T$ that contains $S$ this component induces a missing face with $v_O$ internal vertices. Therefore, $v_O \le k/2$ since $T$ is $k/2$-simple. 

Otherwise, $\Phi <0$ and $2\Phi$ is an integer, hence $\Phi\le -1/2$.

$2.$ We compute the Euler characteristic of the complex $X_S$ in two standard ways:
\[
v_\partial+v_I+3-e_0-e_1-e_2+|S| = \beta_0-\beta_1.
\]
In addition, $3|S|=e_1+2\cdot e_2$ by counting incidences between faces and edges of $S$. The proof is concluded by canceling out $e_2$.

$3.$ We observe that $v_\partial\le e_1$ since every boundary vertex is contained in at least two edges of degree $1$. In addition, as $\beta_0>0$ and $e_0\ge 0$, by inequality (\ref{eqn:white}) there holds $\Phi \le 1-e_1/6$ and the claim follows.
\end{proof}

Let $$\mathcal S(P):=\{ S~:~P(S)=P, V_\partial(S)=\{4,...,v_\partial+3\} \mbox{ and }V_I(S)=\{v_\partial+4,\ldots,v_\partial+v_I+3\} \}$$ where $P$ is an admissible tuple of parameters. Note that up to a factor of $\binom {n-3} {v_\partial,v_I,n-3-v_\partial-v_I}$ choices for the labels of the vertices, $\mathcal S(P)$ contains all the sets $S$ with $P(S)=P$. 

In addition, suppose that $S$ is admissible with vertex sets $ V_\partial(S)=\{4,...,v_\partial+3\}$ and $V_I(S)=\{v_\partial+4,\ldots,v_\partial+v_I+3\}$. Let $\mathcal T_k(S)$ denote the set of $k/2$-simple triangulations $T$ with vertex set $V(T)=\{1,...,k+3\}$ that contain $S$. We observe that the number of all triangulations $T\in\mathcal T$ that contain $S$ is equal to $\binom {n-|V(S)|}{v_O}\cdot |\mathcal T_k(S)|$.

\begin{lemma}
\label{lem:theBigBounds}
Let $P=(v_\partial,v_I,v_O,\Phi)\in\mathcal P$ be an admissible tuple of parameters and $S\in\mathcal S(P)$. Then,
\begin{enumerate}
\item
\[
|\mathcal T_k(S)| \le v_O!\gamma^{v_O} k^{O(v_\partial)},
\]
\item
\[
|\mathcal S(P)| \le v_I!\gamma^{v_I}k^{O(v_\partial)}.
\]
\end{enumerate}
\end{lemma}
\begin{proof}
$1.$ We claim that $|\mathcal T_k(S)|\le t_{v_\partial+v_O}$. Indeed, consider a triangulation $T\in\mathcal T_k(S)$ and suppose that we remove the $v_I$ internal vertices of $S$ and replace the faces of $S$ with a triangulation $S'$ of the area it covers that uses only the $v_\partial$ boundary vertices and $1,2,3$. This procedure yields a triangulation $ (T\setminus S)\cup S'$ of $[123]$ with $v_\partial+v_O$ internal vertices. In addition, note that this procedure does not depend on the particular planar embedding of $S$ that $T$ induces since $S'$ can be chosen independently of the embedding. Therefore, for a fixed $S'$, this procedure is a well defined injection from $\mathcal T_k(S)$ to the triangulations of $[123]$ with $v_\partial+v_O$ internal vertices. Consequently,
\[
|\mathcal T_k(S)|\le (v_\partial+v_O)!\gamma^{v_\partial+v_O} \le v_O!(v_\partial+v_O)^{v_\partial}\cdot\gamma^{v_\partial+v_O} \le
v_O!\gamma^{v_O} k^{O(v_\partial)}.
\]

$2.$ We describe a procedure to construct an arbitrary set $S\in\mathcal S(P)$, i.e., $S$ is a contained in some triangulation of $[123]$, and bound the number of possible such procedures. We first choose an Eulerian graph $G$ with $v_\partial+3$ vertices such that $G\cup [123]$ is planar. Note that the faces of an Eulerian graph that is embedded in the plane have a proper $2$-coloring. We embed $G\cup [123]$ in the plane such that $[123]$ bounds the outer face, and color the internal faces in green and white such that every two faces that share an edge of $G$ have different colors. In addition, we triangulate the green area with $v_I$ internal vertices. The constructed $S$ is the set of triangles that participate in the triangulation of the green area. 
Every set $S\in\mathcal S(P)$ can be constructed by this method by letting $G$ be the graph whose vertex set is $V_{\partial}(S)\cup\{1,2,3\}$ and whose edge set is the edges of degree $1$ in $S$ (See Figure \ref{fig:procedure}).

\begin{figure}[ht]
\begin{subfigure}[b]{0.45\textwidth}
\begin{center}
\begin{tikzpicture}
\node[circle,draw,fill=gray!50,inner sep=0.2pt] (1) at (1.5,-1.5){\footnotesize $1$};
\node[circle,draw,fill=gray!50,inner sep=0.2pt] (2) at (3,-3){\footnotesize $2$};
\node[circle,draw,fill=gray!50,inner sep=0.2pt] (3) at (5,-3){\footnotesize $3$};
\node[circle,draw,fill=gray!50,inner sep=0.2pt] (4) at (4.5,-1.5){\footnotesize $4$};
\node[circle,draw,fill=gray!50,inner sep=0.2pt] (5) at (4.5,0){\footnotesize $5$};
\node[circle,draw,fill=gray!50,inner sep=0.2pt] (6) at (3,0){\footnotesize $6$};
\node[circle,draw,fill=gray!50,inner sep=0.2pt] (7) at (5.5,-2){\tiny $A$};
\node[circle,draw,fill=gray!50,inner sep=0.2pt] (8) at (5,-0.5){\footnotesize $8$};
\node[circle,draw,fill=gray!50,inner sep=0.2pt] (9) at (6.8,-1){\footnotesize $9$};
\node[circle,draw,fill=gray!50,inner sep=0.2pt] (10) at (6,-.05){\footnotesize $7$};
\node[circle,draw,fill=gray!50,inner sep=0.2pt] (11) at (6,-3){\scriptsize $B$};
\node[circle,draw,fill=gray!50,inner sep=0.2pt] (12) at (7,-3){\scriptsize $C$};
\node[circle,draw,fill=gray!50,inner sep=0.2pt] (13) at (6.50,-2){\scriptsize $D$};
\draw (1)--(2)--(4)--(6)--(2)--(3)--(7)--(4)--(5)--(6)--(1);
\draw (8)--(9)--(10)--(8);
\draw (11)--(12)--(13)--(11);
\end{tikzpicture}
\end{center}
\caption{}
\end{subfigure}
\begin{subfigure}[b]{0.45\textwidth}
\begin{center}
\begin{tikzpicture}
\node[circle,draw,fill=gray!50,inner sep=0.2pt] (1) at (0,1){\footnotesize $1$};
\node[circle,draw,fill=gray!50,inner sep=0.2pt] (2) at (-3,-3){\footnotesize $2$};
\node[circle,draw,fill=gray!50,inner sep=0.2pt] (3) at (3,-3){\footnotesize $3$};
\node[circle,draw,fill=gray!50,inner sep=0.2pt] (4) at (-0,-2.2){\footnotesize $4$};
\node[circle,draw,fill=gray!50,inner sep=0.2pt] (5) at (0.3,0){\footnotesize $5$};
\node[circle,draw,fill=gray!50,inner sep=0.2pt] (6) at (-2,-2.5){\footnotesize $6$};
\node[circle,draw,fill=gray!50,inner sep=0.2pt] (7) at (1.5,-2.3){\tiny $A$};
\node[circle,draw,fill=gray!50,inner sep=0.2pt] (8) at (-0.5,-1.8){\footnotesize $8$};
\node[circle,draw,fill=gray!50,inner sep=0.2pt] (9) at (-.1,-.8){\footnotesize $9$};
\node[circle,draw,fill=gray!50,inner sep=0.2pt] (10) at (-1.1,-2.1){\footnotesize $7$};
\node[circle,draw,fill=gray!50,inner sep=0.2pt] (11) at (.5,-1.9){\scriptsize $B$};
\node[circle,draw,fill=gray!50,inner sep=0.2pt] (12) at (1.2,-1.5){\scriptsize $C$};
\node[circle,draw,fill=gray!50,inner sep=0.2pt] (13) at (.5,-1){\scriptsize $D$};
\draw (11)--(12)--(13)--(11);
\draw (1)--(2)--(4)--(6)--(2)--(3)--(7)--(4)--(5)--(6)--(1);
\draw[dashed] (1)--(3);
\draw (8)--(9)--(10)--(8);
\end{tikzpicture}
\end{center}
\caption{}
\end{subfigure}
\\

\begin{subfigure}[b]{0.45\textwidth}
\begin{center}
\begin{tikzpicture}
\draw[fill=white] (0,1) to (-2,-2.5) to (0.3,0) to (0,-2.2) to (-1.1,-2.1) --(3,-3)--(0,1);
\draw[pattern=dots,pattern color=white] (0,1) to (-2,-2.5) to (0.3,0) to (0,-2.2) to (-1.1,-2.1) --(3,-3)--(0,1);
\draw[fill=white] (-3,-3) to (-2,-2.5) to (0,-2.2) to(-3,-3);
\draw[pattern=dots,pattern color=white] (-3,-3) to (-2,-2.5) to (0,-2.2) to(-3,-3);
\draw[fill=black!30!green] (.5,-1.9) to (1.2,-1.5) to (.5,-1) to (.5,-1.9);
\draw[fill=black!30!green] (0,1) to (-3,-3) to (-2,-2.5) to (0,1);
\draw[fill=black!30!green]  (-3,-3) to (-0,-2.2) to (1.5,-2.3)to (3,-3)to(-3,-3);
\draw[fill=black!30!green] (0,-2.2) to (0.3,0) to (-2,-2.5) to (0,-2.2);
\draw[fill=white] (-0.5,-1.8) to (-.1,-.8) to (-1.1,-2.1) to (-0.5,-1.8);
\draw[pattern=dots,pattern color=white] (-0.5,-1.8) to (-.1,-.8) to (-1.1,-2.1) to (-0.5,-1.8);

\node[circle,draw,fill=gray!50,inner sep=0.2pt] (1) at (0,1){\footnotesize $1$};
\node[circle,draw,fill=gray!50,inner sep=0.2pt] (2) at (-3,-3){\footnotesize $2$};
\node[circle,draw,fill=gray!50,inner sep=0.2pt] (3) at (3,-3){\footnotesize $3$};
\node[circle,draw,fill=gray!50,inner sep=0.2pt] (4) at (-0,-2.2){\footnotesize $4$};
\node[circle,draw,fill=gray!50,inner sep=0.2pt] (5) at (0.3,0){\footnotesize $5$};
\node[circle,draw,fill=gray!50,inner sep=0.2pt] (6) at (-2,-2.5){\footnotesize $6$};
\node[circle,draw,fill=gray!50,inner sep=0.2pt] (7) at (1.5,-2.3){\tiny $A$};
\node[circle,draw,fill=gray!50,inner sep=0.2pt] (8) at (-0.5,-1.8){\footnotesize $8$};
\node[circle,draw,fill=gray!50,inner sep=0.2pt] (9) at (-.1,-.8){\footnotesize $9$};
\node[circle,draw,fill=gray!50,inner sep=0.2pt] (10) at (-1.1,-2.1){\footnotesize $7$};
\node[circle,draw,fill=gray!50,inner sep=0.2pt] (11) at (.5,-1.9){\scriptsize $B$};
\node[circle,draw,fill=gray!50,inner sep=0.2pt] (12) at (1.2,-1.5){\scriptsize $C$};
\node[circle,draw,fill=gray!50,inner sep=0.2pt] (13) at (.5,-1){\scriptsize $D$};

\draw (1)--(2)--(4)--(6)--(2)--(3)--(7)--(4)--(5)--(6)--(1);
\draw[dashed] (1)--(3);
\draw (8)--(9)--(10)--(8);
\end{tikzpicture}
\end{center}
\caption{}
\end{subfigure}
\begin{subfigure}[b]{0.45\textwidth}
\begin{center}
\begin{tikzpicture}
\draw[fill=white] (0,1) to (-2,-2.5) to (0.3,0) to (0,-2.2) to (-1.1,-2.1) --(3,-3)--(0,1);
\draw[pattern=dots,pattern color=white] (0,1) to (-2,-2.5) to (0.3,0) to (0,-2.2) to (-1.1,-2.1) --(3,-3)--(0,1);
\draw[fill=white] (-3,-3) to (-2,-2.5) to (0,-2.2) to(-3,-3);
\draw[pattern=dots,pattern color=white] (-3,-3) to (-2,-2.5) to (0,-2.2) to(-3,-3);
\draw[fill=black!30!green] (.5,-1.9) to (1.2,-1.5) to (.5,-1) to (.5,-1.9);
\draw[fill=black!30!green] (0,1) to (-3,-3) to (-2,-2.5) to (0,1);
\draw[fill=black!30!green]  (-3,-3) to (-0,-2.2) to (1.5,-2.3)to (3,-3)to(-3,-3);
\draw[fill=black!30!green] (0,-2.2) to (0.3,0) to (-2,-2.5) to (0,-2.2);
\draw[fill=white] (-0.5,-1.8) to (-.1,-.8) to (-1.1,-2.1) to (-0.5,-1.8);
\draw[pattern=dots,pattern color=white] (-0.5,-1.8) to (-.1,-.8) to (-1.1,-2.1) to (-0.5,-1.8);

\node[circle,draw,fill=gray!50,inner sep=0.2pt] (1) at (0,1){\footnotesize $1$};
\node[circle,draw,fill=gray!50,inner sep=0.2pt] (2) at (-3,-3){\footnotesize $2$};
\node[circle,draw,fill=gray!50,inner sep=0.2pt] (3) at (3,-3){\footnotesize $3$};
\node[circle,draw,fill=gray!50,inner sep=0.2pt] (4) at (-0,-2.2){\footnotesize $4$};
\node[circle,draw,fill=gray!50,inner sep=0.2pt] (5) at (0.3,0){\footnotesize $5$};
\node[circle,draw,fill=gray!50,inner sep=0.2pt] (6) at (-2,-2.5){\footnotesize $6$};
\node[circle,draw,fill=gray!50,inner sep=0.2pt] (7) at (1.5,-2.3){\tiny $A$};
\node[circle,draw,fill=gray!50,inner sep=0.2pt] (8) at (-0.5,-1.8){\footnotesize $8$};
\node[circle,draw,fill=gray!50,inner sep=0.2pt] (9) at (-.1,-.8){\footnotesize $9$};
\node[circle,draw,fill=gray!50,inner sep=0.2pt] (10) at (-1.1,-2.1){\footnotesize $7$};
\node[circle,draw,fill=gray!50,inner sep=0.2pt] (11) at (.5,-1.9){\scriptsize $B$};
\node[circle,draw,fill=gray!50,inner sep=0.2pt] (12) at (1.2,-1.5){\scriptsize $C$};
\node[circle,draw,fill=gray!50,inner sep=0.2pt] (13) at (.5,-1){\scriptsize $D$};
\node[circle,draw,fill=gray!50,inner sep=0.2pt] (14) at (0.5,-2.7){\scriptsize $E$};

\draw (1)--(2)--(4)--(6)--(2)--(3)--(7)--(4)--(5)--(6)--(1);
\draw[dashed] (1)--(3);
\draw (8)--(9)--(10)--(8);
\draw (2)--(14)--(4);
\draw (3)--(14)--(7);
\draw (5)--(9);
\draw (4)--(8);
\draw (9)--(6)--(10);
\draw (9)--(4)--(10);
\end{tikzpicture}
\end{center}
\caption{}
\end{subfigure}
\caption{Illustration of the procedure in the proof of Lemma \ref{lem:theBigBounds} item 2. (a) An Eulerian planar graph $G$. (b) An embedding of $G\cup [123]$. The edge $13$ is dashed to emphasize that it is not in $G$. (c) A green / white coloring of the faces. (d) A triangulation of the green area with one internal vertex $E$. The constructed set is $S$ from Figure \ref{fig:example}. It is clear that this procedure is reversible, so every admissible set $S\in \mathcal S(P)$ can be constructed.}
\label{fig:procedure}
\end{figure}
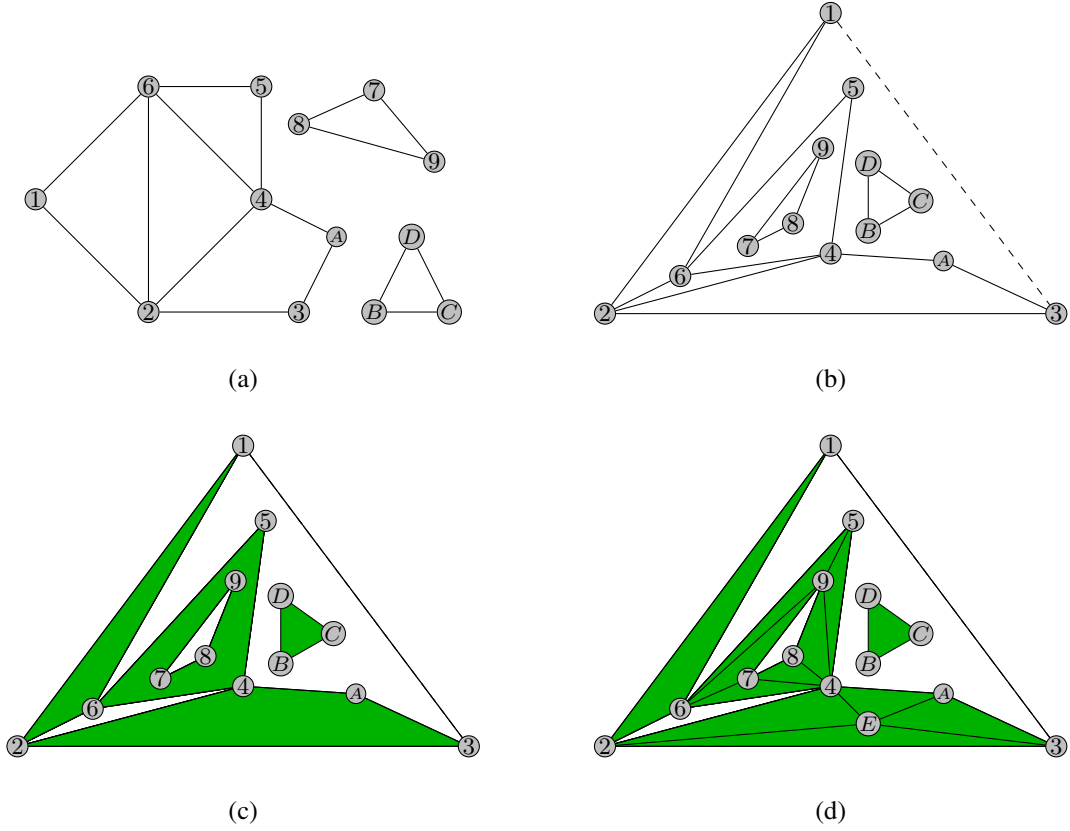

It remains to bound the number of choices made during the procedure by $v_I!\gamma^{v_I}k^{O(v_\partial)}.$ 
First, There are $(v_\partial+3)!\exp{[O(v_\partial)]}\le k^{O(v_\partial)}$ planar graphs with $v_\partial+3$ vertices \cite{planar_enum}. 
Second, we crudely bound by $v_\partial^{O(v_\partial)}$ the number of planar embeddings of the $(v_\partial+3)$-vertex graph $G\cup[123]$ in which $[123]$ bounds the outer face. Indeed, any embedding can be realized by the following scheme.
\begin{enumerate}[(I)]
\item Order the connected components $C_1,...,C_m$ of $G \cup [123]$, where $C_1$ is the component of $[123]$. 
\item Embed $C_1$ such that $[123]$ bounds the outer face.
\item In every step $1<i\le m$: choose an inner face of the currently embedded graph $G_{i-1}=C_1\cup\cdots\cup C_{i-1}$, and embed $C_i$ inside a disk that is contained in the chosen face but does not intersect the embedded $G_{i-1}$. 
\end{enumerate}

We claim that there are at most $v_\partial^{O(v_\partial)}$ possible schemes. This holds for item (I) immediately. 
In addition, we choose a planar embedding of every component $C_i$. Note that the embedding of a connected graph into the plane is determined by choosing the {oriented boundary of the outer face} and a {cyclic ordering of the neighbors of every vertex}, that is also known as a {\em rotation system} \cite{graphs_on_surfaces}. Moreover, listing all the desired oriented outer faces of all the components can be done in at most $(v_\partial +3)!$ ways since each vertex belongs to at most one outer face. In addition, the number of rotation systems is at most $(2|E|)!=v_\partial^{O(v_\partial)}$ where $E$ is edge set of $G\cup[123]$. In item (III) we select a face $f$ of $G_{i-1}$ and embed $C_i$ in a disk that is contained in $f$ and does not intersect $G_{i-1}$. All such disks yield equivalent embeddings so we only need to choose the face $f$, and there are at most $O(v_\partial)$ faces and only $m\le v_\partial$ steps. 

Note that in order to realize a specific embedding, one may need to select the ordering of $C_1,...,C_m$ carefully. For instance, in order to create an annulus shaped area whose outer and inner cycles are $C_i$ and $C_j$ respectively, the ordering must satisfy $i<j$.  

Third, once $G\cup[123]$ is embedded in the plane, there are only two ways to color its faces. 
Finally, the number of ways to triangulate the green part with $v_I$ internal vertices is at most $v_I!\gamma^{v_I}k^{O(v_\partial)}$, by the same argument as in item $1.$ Namely, every triangulation of the green part with $v_I$ internal vertices induces a triangulation of $[123]$ with $v_\partial+v_I$ internal vertices, by adding it to a fixed triangulation of the white part.
\end{proof}

We conclude this section with the proof of Lemma \ref{lem:key_lemma}.
\begin{proof}[Proof of Lemma \ref{lem:key_lemma}.]
Recall that  $p=c/\sqrt{n}$ where $c>1/\sqrt{\gamma}$. Denote $\lambda := 1/(\gamma c^2)<1$ thus $p^{-2}=\lambda\gamma n$. By Claim \ref{clm:trick},
\begin{align}
\nonumber
\frac{\Delta}{\mu^2} ~\le~& \frac{1}{|\mathcal T|^2} \sum_{\substack{S\subset \binom {[n]}3\\0<|S|<2k+1}} p^{-|S|}|\{T\in \mathcal T~:~T\supset S \}|^2\\
~=~& 
\label{aleq:first}
\frac{k^{O(1)}}{(\gamma n)^{2k}} \sum_{P\in\mathcal P}\sum_{\substack{S:\\P(S)=P}} p^{-|S|}|\{T\in \mathcal T~:~T\supset S \}|^2\\
~\le~& 
\label{aleq:3}
 \frac{k^{O(1)}}{(\gamma n)^{2k}} \sum_{P\in\mathcal P}\frac{n^{v_\partial+v_I}}{v_I!}\sum_{S\in\mathcal{S}(P)} p^{-|S|}
 \left(\frac{n^{v_O}}{v_O!}|\mathcal T_k(S)|\right)^2\\
 ~\le~&
 \label{aleq:4}
  \frac{1}{(\gamma n)^{2k}} \sum_{P\in\mathcal P}k^{O(v_\partial)}\frac{n^{v_\partial+v_I}}{v_I!}\sum_{S\in\mathcal{S}(P)} (\lambda \gamma n)^{v_\partial+v_I+\Phi} (\gamma n)^{2v_O}
  \\
  \label{aleq:5}
 ~=~&
  \frac{1}{(\gamma n)^{2k}} \sum_{P\in\mathcal P}
k^{O(v_\partial)}(\lambda \gamma n)^{\Phi}
\lambda^{v_\partial+v_I}
  \frac{n^{v_\partial+v_I}}{v_I!}|\mathcal{S}(P)| (\gamma n)^{v_\partial+v_I+2v_o}
  \\
  \label{aleq:6}
 ~\le~&
    \sum_{P\in\mathcal P}k^{O(v_\partial)}
(\lambda \gamma n)^{\Phi}
\lambda^{v_\partial+v_I}(\gamma n)^{2(v_\partial+v_I+v_O)-2k}
\\
\nonumber
   ~=~&
    \sum_{P\in\mathcal P}k^{O(v_\partial)}
(\lambda \gamma n)^{\Phi}
\lambda^{v_\partial+v_I}.
\end{align}
The first two transitions (\ref{aleq:first}) and (\ref{aleq:3}) are derived by taking into account the canonical $n$-labellings of all considered triangulations. For instance, Lemma \ref{lem:manysimple} asserts that the number of $k/2$-simple triangulations with $k$ labeled internal vertices is $k!\gamma^k/k^{O(1)}$. Therefore, $|\mathcal T|\ge (\gamma n)^k/k^{O(1)}$ since there are $\binom {n-3}k\approx n^k/k!$ proper $n$-labellings of the $k$ vertices. In addition, an arbitrary set of faces $S$ having a tuple of parameters  $P(S)=P$ is obtained by choosing a set in $\mathcal S(P)$ and a proper $n$-labeling for the $v_\partial + v_I$ vertices. Similarly, a triangulation $S\subset T\in\mathcal T$ is comprised of a triangulation in $\mathcal T_k(S)$ and a proper $n$-labeling of the $v_O$ outer vertices.
We obtain (\ref{aleq:4}) by applying the second part of Claim \ref{clm:phieuler} and the first part of Lemma \ref{lem:theBigBounds}. Note that we group all the $k^{O(v_\partial)}$ factors together.
Equation (\ref{aleq:5}) is straightforward and inequality 
(\ref{aleq:6}) follows from the second part of Lemma \ref{lem:theBigBounds}.

The proof is concluded by splitting the summation to the three following cases. The first part of Claim \ref{clm:phieuler} shows that these cases indeed cover the entire summation. The fact that the size of $\mathcal P$ is of order $k^{O(1)}$ is important to the analysis of all the cases below.
\begin{enumerate}
\item {\em $v_\partial\ge 13$.} 
By the third item of Claim \ref{clm:phieuler}, $\Phi\le 1 - \frac {v_\partial}6$. Consequently, the contribution of this case to the sum is bounded by 
\[
k^{O(v_\partial)} (\lambda \gamma n)^{1-v_\partial/6} \le
\lambda\gamma n \left(\frac{k^{O(1)}}{n^{1/6}} \right)^{v_\partial}\le \frac 1n.
\]

\item {\em $\Phi\le -1/2$ and $v_\partial< 13$.} Since $|\mathcal P|=k^{O(1)}$ and $v_\partial$ is bounded, the contribution of this case to the sum is at most
$$
k^{O(1)}\cdot (\lambda \gamma n)^{-1/2}\le \frac{(\log n)^{O(1)}}{\sqrt{n}}.
$$
\item {$\Phi=0$ and $v_\partial+v_I>k/2$}. By the proof of Claim \ref{clm:phieuler}, $\Phi=0$ only if $v_\partial \le 3$. Therefore the contribution of this case to the sum is bounded by $$k^{O(1)}\lambda^{k/2} \le \frac 1n,$$ since $k\ge \log^2 n$.
\end{enumerate}
\end{proof}
\begin{remark}
\begin{enumerate}
\item Note that the contribution of the third case to the sum is negligible only because we restrict to $k/2$-simple triangulations. This is the part of the argument that breaks down when all triangulations are considered as we saw in Section \ref{sec:sketch}.
\item The choice $k=\log^2n$ could be replaced by any $C\log n < k < n^{1/C}$  for a sufficiently large constant $C$.
\end{enumerate}
\end{remark}

\section{Discussion and open questions}
\label{sec:open}
\begin{itemize}
\item The main open question that this paper suggests is Conjecture \ref{conj:main}. In fact, any lower bound of $p=\Omega(1/\sqrt{n})$ for the threshold probability for simple connectivity would be of great interest. The main challenge is that a contraction of a cycle that is not a triangulation can use a face of $Y_2(n,p)$ more than one time. In consequence, contractions of a given cycle are tricky combinatorial objects, and it is not clear how to bound the probability that they appear in $Y_2(n,p)$. Babson, Hoffman and Kahle solved this problem by showing that $Y_2(n,p)$ is a.a.s hyperbolic if $p<n^{-1/2-\varepsilon}$ ~\cite{BHK}. Can their bound be derived by a pure combinatorial argument?

\item It is natural to compare the random group $\pi_1(Y_2(n,p))$ with {\.Z}uk's random triangular model \cite{Zuk}. This group has $n$ generators and a random set of relations of length $3$. The vanishing of the random triangular group admits a sharp threshold, but the precise threshold probability is not known \cite{ALS,AFL}. Similarly, proving that simple connectivity of $Y_2(n,p)$ admits a sharp threshold, even without knowing the precise threshold probability, is also of interest. On the other hand, it is intriguing to see whether our methods can be adapted to improve the known bounds for the vanishing of the random triangular group.

\item The proof of Theorem \ref{thm:main} implies that $p=1/\sqrt{\gamma n}$ is the sharp threshold probability for the property that every face of $Y=Y_2(n,p)$ belongs to a triangulated sphere in $Y$. What is the threshold probability for the analogous property in higher dimensions? What is the threshold probability if the sphere is replaced by some other manifold? 



\item Computational aspects of this problem can also be considered. A {\em certifier for simple connectivity} is an algorithm $\mathcal A$ whose input is a $2$-dimensional simplicial complex $Y$ and its output is {\em 'YES'} only if $Y$ is simply connected.
Let $c>1/\sqrt{\gamma}$. We ask: is there a {\em polynomial time} certifier for simple connectivity whose probability to output {\em 'YES'} on a $Y_2(n,c/\sqrt{n})$-distributed input is $1-o(1)$?

Note that our proof of Theorem \ref{thm:main} yields a certifier that runs in quasipolynomial time of $n^{O(\log n)}$ by going over all the triangulations with $k=C\log n$ internal vertices of all the cycles of length $3$. In comparison, we observe that the answer is positive if $c>1/2$ since the previous bound for simple connectivity of $Y_2(n,p)$ is based on a polynomial time certifier~\cite{KPS}.

\end{itemize}

\end{document}